\documentclass[fleqn,preprint,3p,a4paper]{elsarticle}

\usepackage{amssymb}
\usepackage{amsmath}
\usepackage{amsthm}
\usepackage{dcolumn}
\usepackage{endnotes}
\usepackage{tabularx}
\usepackage[matrix,arrow]{xy}
\usepackage{wasysym}

\newtheorem{theorem}{Theorem}[section]
\newtheorem{proposition}[theorem]{Proposition}
\newtheorem{lemma}[theorem]{Lemma}
\newtheorem{corollary}[theorem]{Corollary}

\theoremstyle{definition}
\newtheorem{definition}[theorem]{Definition}

\newtheorem{remark}[theorem]{Remark}

\newcommand{\cl}{{\rm cl}}
\newcommand{\ii}{{\rm int}}

\journal{}

\begin{document}

\begin{frontmatter}

\title{Countable quasicontinuous domains are quasialgebraic\tnoteref{t1}}
\tnotetext[t1]{This research was supported by the National Natural Science Foundation of China (Nos. 12071199, 11661057).}

\author[X. Xu]{Xiaoquan Xu}
\ead{xiqxu2002@163.com}
\address[X. Xu]{Fujian Key Laboratory of Granular Computing and Applications, Minnan Normal University, Zhangzhou 363000, China}

\begin{abstract}
We prove that every quasicontinuous domain that fails to be quasialgebraic admits the unit interval [0, 1]
as its monotone Lawson-continuous image. As a result, every countable quasicontinuous domain is quasialgebraic.
\end{abstract}


\end{frontmatter}

\section{Introduction}

In his pioneering work in what has come to be called ``domain theory" Dana Scott introduced and studied a crucial dcpos
which came to be called the continuous domains, which are appropriate for models of various typed and untyped lambda-calculi
and functional programming languages (see \cite{Abramsky-Jung-1994, GHKLMS-2003, Goubault-2013, Streicher-2006}). One of the main research directions of domain theory is to carry as much as possible of the theory of continuous domains and their algebraic forms to as general an ordered structure as possible. Two of the most successful generalizations have been quasicontinuous domains and quasialgebraic domains. They have attracted a considerable deal of attention (see \cite{Erne-2009, Erne-2018-2, GHKLMS-2003, Gierz-Lawson-Stralka-1983, Goubault-2013, Heckmann-1992, Heckmann-Keimel-2013, Jia-2018, Xu-2016-2}).

As we all know, every (quasi)algebraic domain is (quasi)continuous, but the converse is not true in general. For example, the unit interval [0, 1] with the usual order is a (quasi)continuous domain but not (quasi)algebraic. In \cite{Jia-Li-Luan-2025}, Jia et al. proved that every continuous domain that fails to be algebraic admits the unit interval $[0, 1]$ as its Scott-continuous retract. As a result, every countable domain is algebraic.

In this paper, drawing lessons from the main technique used in the proof of Urysohn's lemma in topology, we prove that if $P$ is a quasicontinuous domain that fails to be quasialgebraic, then there is a surjective monotone Lawson continuous function $f : P \rightarrow [0, 1]$. One immediate corollary is that every countable quasicontinuous domain is quasialgebraic.

\section{Quasicontinuous domains and quasialgebraic domains}

Firsr, we briefly recall some fundamental concepts and basic results about (quasi)continuous domains and (quasi)algebraic domains. For further details, we refer the reader to \cite{GHKLMS-2003, Goubault-2013}.

For a poset $P$ and $A\subseteq P$, let
$\mathord{\downarrow}A=\{x\in P: x\leq  a \mbox{ for some }
a\in A\}$ and $\mathord{\uparrow}A=\{x\in P: x\geq  a \mbox{
	for some } a\in A\}$. For  $x\in P$, we write
$\mathord{\downarrow}x$ for $\mathord{\downarrow}\{x\}$ and
$\mathord{\uparrow}x$ for $\mathord{\uparrow}\{x\}$. The set $A$
is called a \emph{lower set} (resp., an \emph{upper set}) if
$A=\mathord{\downarrow}A$ (resp., $A=\mathord{\uparrow}A$). Let $P^{(<\omega)}=\{F\subseteq P : F \mbox{~is a nonempty finite set}\}$ and $\mathbf{Fin} P=\{\uparrow F :$ $F\in P^{(<\omega)}\}$.
 For a nonempty subset $C$ of $P$, define $\mathrm{min}(C)=\{c\in C : c \mbox{~ is a minimal element of~} C\}$. The set of all natural numbers is denoted by $\mathbb{N}$. For a set $X$, let $|X|$ be the cardinality of $X$ and let $\omega=|\mathbb{N}|$.

A nonempty subset $D$ of a poset $P$ is \emph{directed} if every two
elements in $D$ have an upper bound in $D$. The set of all directed sets of $P$ is denoted by $\mathcal D(P)$. The poset $P$ is called a \emph{directed complete poset}, or \emph{dcpo} for short, if for any
$D\in \mathcal D(P)$, its supremum $\vee D$ exists in $P$.

A subset $U$ of a dcpo $P$ is said to be \emph{Scott open} if (i) $U=\mathord{\uparrow}U$, and (ii) $\vee D\in U$ implies $D\cap
U\neq\emptyset$ for any directed subset $D$. All Scott open subsets of $P$ form a topology,
called the \emph{Scott topology} on $P$ and
denoted by $\sigma(P)$. The space $\Sigma  P=(P,\sigma(P))$ is called the
\emph{Scott space} of $P$. The \emph{lower topology} on $P$, generated
by $\{P\setminus \uparrow x : x\in P\}$ (as a subbase), is denoted by $\omega (P)$. Dually, define the \emph{upper topology} on $P$ and denote it by $\upsilon(P)$. The topology generated
by $\omega (P)\bigcup\sigma (P)$ is called the \emph{Lawson topology} on $P$ and is denoted by $\lambda (P)$. The interval topology $\theta(P)$ on $P$ is the one generated by $\omega (P)\bigcup\upsilon (P)$. For the unit interval $[0, 1]$, it is easy to verify that $\lambda([0, 1])=\theta([0, 1])$. For two dcpos $P$ and $Q$, a mapping $f : P \rightarrow Q$ is called \emph{Lawson continuous}, if $f : (P, \lambda(P)) \longrightarrow (Q, \lambda(Q))$ is continuous.

The following is a key feature of the Scott topology (see \cite[Proposition II-2.1]{GHKLMS-2003}).

\begin{lemma}\label{lem-Scott-cont-charac}  Let $P, Q$ be dcpos and $f : P \longrightarrow Q$. Then the following two conditions are equivalent:
\begin{enumerate}[\rm (1)]
	\item $f$ is Scott continuous, that is, $f : \Sigma  P \longrightarrow \Sigma  Q$ is continuous.
	\item For any $D\in \mathcal D(P)$, $f(\vee D)=\vee f(D)$.
\end{enumerate}
\end{lemma}

Rudin's Lemma, given by Mary Rudin in \cite{Rudin-1981}, is a useful tool in non-Hausdorff topology and plays a crucial role in domain theory (see \cite{Erne-2009, Erne-2018-2, GHKLMS-2003, Gierz-Lawson-Stralka-1983, Goubault-2013, Heckmann-1992, Heckmann-Keimel-2013, Jia-2018, Jung-1989, Xu-2016-2}). In this paper we will use the following Jung's version of Rudin's Lemma (see \cite[Theorem 4.11]{Jung-1989}, \cite[Lemma III-3.3]{GHKLMS-2003}, or \cite[Corollary 3.5]{Heckmann-Keimel-2013}).

\begin{lemma}\label{lem-Rudin-lemma}  Let $P$ be a poset and $\{F_i : i\in I\}\subseteq P^{(<\omega)}$. If $\{\uparrow F_i : i\in I\}$ is a filtered family, then there is a directed subset $D$ of $\bigcup_{i\in I}F_i$ that meets all $F_i$.
\end{lemma}

As a corollary of Lemma \ref{lem-Rudin-lemma}, we have the following useful result (cf. \cite[Corollary 3.9]{Heckmann-Keimel-2013}).

\begin{corollary}\label{cor-Rudin-lemma}  Let $P$ be a dcpo, $U$ a Scott-open set of $P$ and $\{\uparrow F_i : i\in I\}\subseteq \mathbf{Fin}~P$ be a filtered family with $\bigcap_{i\in I}\uparrow F_i\subseteq U$. Then $\uparrow F_i\subseteq U$ for some $i\in I$.
\end{corollary}

For a dcpo $P$ and $A, B\subseteq P$, we say $A$ is \emph{way below} $B$, written $A\ll B$, if for each $D\in \mathcal D(P)$, $\vee D\in \uparrow B$ implies $D\cap \uparrow A\neq \emptyset$. For $B=\{x\}$, a singleton, $A\ll B$ is
written $A\ll x$ for short. For $x\in P$, let $\Downarrow x = \{u\in P : u\ll x\}$ and $w(x)=\{F\in P^{(<\omega)} : F\ll x\}$. Let $K(P)=\{k\in P : k\ll k\}$. Points in $K(P)$ are called \emph{compact elements} of $P$.

\begin{definition}\label{def-continuous-domain} Let $P$ be a dcpo.
\begin{enumerate}[\rm (1)]
\item $P$ is called a \emph{continuous domain}, if for each $x\in P$, $\Downarrow x$ is directed
and $x=\vee\Downarrow x$.
\item $P$ is called an \emph{algebraic domain}, if for each $x\in P$, $K(P)\cap \downarrow x$ is directed and $x=\vee (K(P)\cap \downarrow x)$.
\item $P$ is called a \emph{quasicontinuous domain}, if for each $x\in P$, $\{\uparrow F : F\in w(x)\}$ is filtered and $\uparrow x=\cap
\{\uparrow F : F\in w(x)\}$.
\item $P$ is called a \emph{quasialgebraic domain}, if for each $x\in P$, $\{\uparrow F\in \mathbf{Fin} P : x\in \uparrow F, F\ll F\}$ is filtered and $\uparrow x=\cap \{\uparrow F\in \mathbf{Fin} P : x\in \uparrow F, F\ll F\}$.
\end{enumerate}
\end{definition}

\begin{proposition}\label{prop-continuous-charact} (\cite[Proposition II-2.11 and Theorem II-2.14]{GHKLMS-2003}) For a dcpo $P$, the following two conditions are equivalent:
\begin{enumerate}[\rm (1)]
\item $P$ is continuous.
\item For any $x\in P$ and any $U\in \sigma (P)$, there is $u\in P$ such that $x\in \ii_{\sigma(P)}\uparrow u\subseteq \uparrow u\subseteq U$.
\end{enumerate}
\end{proposition}

\begin{proposition}\label{prop-algebraic-charact} (\cite[Corollary II-2.15]{GHKLMS-2003}) For a dcpo $P$, the following conditions are equivalent:
\begin{enumerate}[\rm (1)]
\item $P$ is algebraic.
\item For any $x\in P$ and any $U\in \sigma (P)$ with $x\in U$, there is $k\in K(P)$ such that $x\in \uparrow k\subseteq U$.
\item The Scott topology $\sigma (P)$ has a basis of sets $\uparrow k$ where $k\in K(P)$.
\end{enumerate}
\end{proposition}

For quasicontinuous domains and quasialgebraic domains, we have the following analogous characterizations.

\begin{proposition}\label{prop-quasicontinuous-charact} (\cite[Proposition II-2.11 and Theorem II-2.14]{GHKLMS-2003}) For a dcpo $P$, the following two conditions are equivalent:
\begin{enumerate}[\rm (1)]
\item $P$ is quasicontinuous.
\item For any $x\in P$ and any $U\in \sigma (P)$ with $x\in U$, there is $F\in P^{(<\omega)}$ such that $x\in \ii_{\sigma(P)}\uparrow F\subseteq \uparrow F\subseteq U$.
\end{enumerate}
\end{proposition}

\begin{proposition}\label{prop-quasialgebraic-charact} (\cite[Exercise II-3.24]{GHKLMS-2003}) For a dcpo $P$, the following conditions are equivalent:
\begin{enumerate}[\rm (1)]
\item $P$ is quasialgebraic.
\item For any $x\in P$ and any $U\in \sigma (P)$ with $x\in U$, there is $F\in P^{(<\omega)}$ such that $x\in \ii_{\sigma(P)}\uparrow F=\uparrow F\subseteq U$.
\item The Scott topology $\sigma (P)$ has a basis of sets $\uparrow F$ where $F\in P^{(<\omega)}$.
\end{enumerate}
\end{proposition}

\begin{definition}\label{def-MC-poset}  (\cite[Definition III-2.1]{GHKLMS-2003}, \cite[Definition 2.2]{Kou-Liu-Luo-2003})  A dcpo $P$ is said to be a \emph{meet-continuous domain} if for any $x\in P$ and any
$D\in \mathcal D(P)$, $x \leq \vee D$ implies $x\in \ii_{\sigma(P)} (\downarrow\cap \downarrow D)$.
\end{definition}

\begin{lemma}\label{lem-MC-poset-charac} (\cite[Theorem 2 and Theorem 3]{Erne-2009}, \cite[Theorem 2.4]{Kou-Liu-Luo-2003}) For a dcpo $P$, the following conditions are equivalent:
\begin{enumerate}[\rm (1)]
\item $P$ is meet-continuous.

\item The lattice of Scott closed sets of $P$ is a complete Heyting algebra.
\item For each $A\subseteq X$ and $x\in X$, $x \in \cl_{\sigma(P)}A$ implies $x\in \cl_{\sigma(P)}(\downarrow x \cap \downarrow A)$.
\item For any $A, B \subseteq P$, $\cl_{\sigma(P)}(\downarrow A\cap \downarrow B)=\cl_{\sigma(P)}\downarrow A \cap \cl_{\sigma(P)} \downarrow B$.
\item For any $U, V\subseteq P$, $\ii_{\sigma(P)} (\uparrow U\cup \uparrow V)=\ii_{\sigma(P)} \uparrow U\cup \ii_{\sigma(P)} \uparrow V$.
\item For $U\in \sigma(P)$ and $x\in X$, $\uparrow (U\cap \downarrow x)\in \sigma(P)$.
\end{enumerate}
\end{lemma}

\begin{proposition}\label{prop-continuous-meet-quasicontinuous} (\cite[Theorem 2.5]{Kou-Liu-Luo-2003}) For a dcpo, the following two conditions are equivalent:
\begin{enumerate}[\rm (1)]
\item $P$ is continuous.
\item $P$ is quasicontinuous and meet-continuous.
\end{enumerate}
\end{proposition}

For algebraic domains, we have the following similar result.

\begin{proposition}\label{prop-algebraic-meet-quasialgebraic}  For a dcpo, the following two conditions are equivalent:
\begin{enumerate}[\rm (1)]
\item $P$ is algebraic.

\item $P$ is quasialgebraic and meet-continuous.
\end{enumerate}
\end{proposition}
\begin{proof} (1) $\Rightarrow$ (2): Suppose that $P$ is an algebraic domain. Then by Proposition \ref{prop-continuous-charact} and Proposition \ref{prop-algebraic-charact}, $P$ is continuous. It follows from Proposition \ref{prop-continuous-meet-quasicontinuous} that $P$ is meet-continuous.

(2) $\Rightarrow$ (1): Assume that $P$ is quasialgebraic and meet-continuous. Let $U\in \sigma (P)$ and $x\in U$. Then by Proposition \ref{prop-quasialgebraic-charact}, there is $F=\{u_1, u_2, ..., u_n\}\in P^{(<\omega)}$ such that $x\in \ii_{\sigma(P)}\uparrow F=\uparrow F\subseteq U$. As $F$ is finite, we have $\uparrow F=\uparrow\mathrm{min}(F)$. Then by Lemma \ref{lem-MC-poset-charac}  $x\in \uparrow \mathrm{min}(F)=\ii_{\sigma (P)}\uparrow \mathrm{min}(F)=\bigcup_{u\in \mathrm{min}(F)} \ii_{\sigma (P)}\uparrow u$. So there is $v\in \mathrm{min}(F)$ such that $x\in \uparrow v\subseteq \bigcup_{u\in \mathrm{min}(F)} \ii_{\sigma (P)}\uparrow u$, whence there is $w\in \mathrm{min}(F)$ satisfying $v\in \ii_{\sigma(P)}\uparrow w\subseteq \uparrow w$. Then $v, w\in \mathrm{min}(F)$ and $w\leq v$, whence $v=m$. Hence $x\in \ii_{\sigma(P)}\uparrow v=\uparrow v\subseteq \uparrow F\subseteq U$. Thus $P$ is algebraic by Proposition \ref{prop-algebraic-charact}.
\end{proof}

\section{The main results}

\begin{theorem}\label{theor-quasicontinuous-not-quasialgebraic} Let $P$ be a quasicontinuous domain that fails to be quasialgebraic. Then there is a surjective monotone Lawson continuous function $f : P \rightarrow [0, 1]$.
\end{theorem}
\begin{proof} First, we define a binary relation $\sqsubseteq$ on $\mathbf{Fin} P$ by $\uparrow F\sqsubseteq \uparrow G$ if and only if $\uparrow F\subseteq \ii_{\sigma(P)}\uparrow G$. As $P$ is quasicontinous, by Proposition \ref{prop-quasicontinuous-charact} the relation $\sqsubseteq$ on $\mathbf{Fin} P$ is idempotent, that is, $\sqsubseteq$ is transitive and satisfies the following interpolation property:
\vskip 1mm
(INT) For all $\uparrow F, \uparrow G\in \mathbf{Fin}F$  with $\uparrow F\sqsubseteq \uparrow G$, there exists $\uparrow H\in \mathbf{Fin} P$  such that $\uparrow F\sqsubseteq \uparrow H\sqsubseteq \uparrow G$.
\vskip 1mm
Second, for any $W\in\sigma(P)$, let $\mathcal F_W=\{\uparrow F\in \mathbf{Fin} P : ~\!\!\uparrow F\in \sigma(P), \uparrow F\subseteq W\}=\{\uparrow F\in \mathbf{Fin} P : ~\!\!\ii_{\sigma(P)}\uparrow F$ $=\uparrow F\subseteq W\}$. As $P$ is not quasialgebraic, by Proposition \ref{prop-quasialgebraic-charact} there is $U\in \sigma(P)\setminus \{\emptyset\}$ such that $U\neq \bigcup \mathcal F_U$. Choose $x\in U\setminus \bigcup \mathcal F_U$. Then by Proposition \ref{prop-quasicontinuous-charact} there is $\uparrow F\in \mathbf{Fin} P$ with $x\in \ii_{\sigma(P)}\uparrow F\subseteq \uparrow F\subseteq U$. Then $\uparrow x\not\in \mathcal F_U, \uparrow F\not\in\mathcal F_U$ (i.e., $\ii_{\sigma (P)}\uparrow x\neq\uparrow x$ and $\ii_{\sigma (P)}\uparrow F\neq\uparrow F$), and $\uparrow x\sqsubseteq \uparrow F$. Let $B=\{\frac{m}{2^n} : m, n\in \mathbb{N}, m\leq 2^n\}$ be the set of dyadic rational numbers in
$[0,1]$. Then by the interpolation property of $\sqsubseteq$ and induction, we can get a family $\{\uparrow F(b) : b\in B\}\subseteq \mathbf{Fin} P$ such that
\vskip 1mm
 (i) $\uparrow F(0)=\uparrow F, \uparrow F(1)=\uparrow x$, and
\vskip 1mm
 (ii) $\uparrow F(b_2)\sqsubseteq \uparrow F(b_1)$ whenever $b_1< b_2$.
\vskip 1mm
For any $b\in B$, we have $\uparrow F(1)\sqsubseteq \uparrow F(b)\sqsubseteq\uparrow F(0)$. So $x\in \ii_{\sigma (P)}\uparrow F(b)\subseteq \uparrow F(b)\subseteq \ii_{\sigma (P)}\uparrow F\subseteq \uparrow F\subseteq U$. As $x\in U\setminus \bigcup \mathcal F_U$, we get the following two conclusions:
\vskip 1mm

(iii) $\ii_{\sigma(P)}\uparrow F(b)\subsetneqq \uparrow F(b)$ for any $b\in B$, and
\vskip 1mm
(iv) $\uparrow F(b_2)\subsetneqq \ii_{\sigma(P)} \uparrow F(b_1)\subsetneqq\uparrow F(b_1)$ whenever $b_1< b_2$.
\vskip 1mm
Define a function $f : P \rightarrow  [0, 1]$ by

$$f(y)=
	\begin{cases}
	\mbox{sup} \{b\in B : b\in \uparrow F(b)\}, & y\in \uparrow F,\\
0,&  otherwise.\\
	\end{cases}$$

As $t=\vee (\downarrow t\cap (B\setminus \{t\})$ for all $t\in [0, 1]$, using (i) and (ii) we can easily verify that $f(y)=\mbox{sup} \{b\in B : b\in \ii_{\sigma(P)}\uparrow F(b)\}$ for any $y\in \uparrow F$. Clearly, $f$ is monotone, $f(x)=1$ and $f(P\setminus \uparrow F)=\{0\}$ (note that $P\setminus \uparrow F\neq\emptyset$ by $\uparrow F\not\in \sigma (P)$).

Now we prove that $f$ is Lawson continuous. For each $\alpha\in (0, 1]$ and each $\beta\in [0, 1)$, we have

$$\begin{array}{lll}
	f^{-1}([0, \alpha))& =\{y\in P: f(y)<\alpha\}\\
	           & =(P\setminus \uparrow F)\cup \{y\in \uparrow F : \mbox{sup}\{b\in B : b\in \uparrow F(b)\}<\alpha\}\\
	           & =(P\setminus \uparrow F)\cup (\uparrow F\cap\bigcup\limits_{b\in B, b<\alpha}(P\setminus \uparrow F(b)))\\
               & =(P\setminus \uparrow F)\cup \bigcup\limits_{b\in B, b<\alpha}(P\setminus \uparrow F(b))\in \omega(p), \mbox{and}\\
	\end{array}$$

$$\begin{array}{lll}
	f^{-1}((\beta, 1])& =\{y\in P : f(y)>\beta\}\\
	           & =\{y\in \uparrow F : \mbox{sup}\{b\in B : b\in \uparrow F(b)\}<\alpha\}\\
	           & =\{y\in \uparrow F : \mbox{sup}\{b\in B : b\in \ii_{\sigma(P)}\uparrow F(b)\}>\beta\}\\
               & =\bigcup\limits_{b\in B, b>\beta}\ii_{\sigma(P)}\uparrow F(b)\in \sigma(P).\\
	\end{array}$$

\noindent Therefore, $f$ is Lawson continuous.

Finally, we show that $f$ is surjective. Let $t\in (0, 1)$. For any $b\in B$ with $b<t$. Choose $b^\prime\in B$ with $b<b^\prime <t$. Then for any $r\in B$ with $t<r$, we have $\ii_{\sigma(P)}\uparrow F(r)\subsetneqq \uparrow F(r)\subsetneqq \ii_{\sigma(P)}\uparrow F(b^\prime)\subsetneqq \uparrow F(b^\prime)\subsetneqq \ii_{\sigma(P)}\uparrow F(b)\subsetneqq\uparrow F(b)$, and hence $f^{-1}((t, 1])=\bigcup\limits_{r\in B, r>t}\uparrow F(b^\prime)\subseteq\ii_{\sigma(P)}\uparrow F(b^\prime)\subsetneqq \uparrow F(b)$. So $f^{-1}((t, 1])\subseteq \bigcap\limits_{b\in B, b<t}\uparrow F(b)$. By (iv) $\{\uparrow F(b) : b\in B, b<t\}\subseteq \mathbf{Fin} P$ is filtered. Then $\bigcap\limits_{b\in B, b<t}\uparrow F(b)\neq f^{-1}((t, 1])$, for otherwise $\bigcap\limits_{b\in B, b<t}\uparrow F(b)= f^{-1}((t, 1])$ would imply $\uparrow F(b)\subseteq f^{-1}((t, 1])$ for some $b\in B$ with $b<t$ by Corollary \ref{cor-Rudin-lemma}, whence $\uparrow F(b)= f^{-1}((t, 1])\in \sigma(P)$, a contradiction. Choose $z\in \bigcap\limits_{b\in B, b<t}\uparrow F(b)\setminus f^{-1}((t, 1])$. Then $f(z)=t$. This completes the proof that $f(P)=[0, 1]$.
\end{proof}

\begin{remark}\label{rem-Uryshon-lemma} Let $X$ be a topological space and $\mathcal O(X)$ be the set of all open subsets of $X$. Define a relation $\sqsubseteq_t$ on $\mathcal O(X)$ by $V\sqsubseteq_t W$ if and only if $\overline{V}\subseteq  W$, where $\overline{V}$ is the closure of $V$ in $X$. Obviously, $\sqsubseteq_t$ is transitive. If $X$ is a normal space, then $\sqsubseteq_t$ satisfies the interpolation property (INT). For a pair $A, B$ of disjoint closed subsets of a normal space $X$, by the normality of $X$, there is an open set $U$ such that $A\subseteq  U \subseteq \overline{U}\subseteq X\setminus B$. The main technique used in the proof of Urysohn's lemma is to construct a family $\{U(b) : b\in B\}\subseteq \mathcal O(X)$ (where $B=\{\frac{m}{2^n} : m, n\in \mathbb{N}, m\leq 2^n\}$) such that

\vskip 1mm
(1) $U(1)=U, U(0)=X\setminus B$, and

\vskip 1mm

(2) $U(b_2)\sqsubseteq_t U(b_1)$ whenever $b_1 <b_2$,

\vskip 1mm

\noindent and define the function $f : X \rightarrow [0, 1]$ by

$$f(y)=
	\begin{cases}
	\mbox{sup} \{b\in B : b\in U(b)\}, & y\in X\setminus B,\\
0,&  otherwise.\\
	\end{cases}$$

\end{remark}

\begin{remark}\label{rem-technique-Urysohn-lemma}
In \cite{Xu-1995}, we first drew on the technique of proof of Urysohn's lemma to present a direct approach to
the construction of fundamental homomorphisms of $M$-continuous
lattices into the unit interval $[0, 1]$ and show that $M$-continuous lattices
admit enough such homomorphisms into $[0, 1]$ to separate points. Later, in \cite{Xu-2001} it was adopted to constructively prove
that if $P$ is a quasicontinuous domain and all lower closed subsets in
$(P, \lambda(P))$ are closed in $(P, \omega(P)))$, then $(P, \lambda(P))$ is strictly completely regular ordered space, which gave a partial answer to a problem posed by Jimmie Lawson in \cite{Lwason-1991}.
\end{remark}

Since $|[0, 1]|=c>\omega$, from Theorem \ref{theor-quasicontinuous-not-quasialgebraic} we directly deduce the following.

\begin{corollary}\label{cor-countable-quasicontinuous-is-quasialgebraic} Every countable quasicontinuous domain is quasialgebraic.
\end{corollary}

Finally, by  Proposition \ref{prop-continuous-meet-quasicontinuous}, Proposition \ref{prop-algebraic-meet-quasialgebraic} and Corollary \ref{cor-countable-quasicontinuous-is-quasialgebraic}, we get the following.

\begin{corollary}\label{cor-countable-continuous-is-algebraic}  (\cite[Corollary 1]{Jia-Li-Luan-2025}) Every countable continuous domain ais algebraic.
\end{corollary}

\end{document}